\documentclass[12pt, reqno]{amsart}
\UseRawInputEncoding
\usepackage{amsmath, amsthm, amscd, amsfonts, amssymb, graphicx, color}
\usepackage[bookmarksnumbered, colorlinks, plainpages]{hyperref}
\hypersetup{colorlinks=true,linkcolor=red, anchorcolor=green,
citecolor=cyan, urlcolor=red, filecolor=magenta, pdftoolbar=true}

\newtheorem{theorem}{Theorem}[section]

\newtheorem{corollary}[theorem]{Corollary}
\theoremstyle{definition}
\newtheorem{definition}[theorem]{Definition}

\theoremstyle{remark}
\newtheorem{remark}[theorem]{Remark}
\numberwithin{equation}{section}

\begin{document}

\setcounter{page}{1}

\title[An extension of the $\rho$-operator radii]
{An extension of the $\rho$-operator radii}

\author[F.~Kittaneh \MakeLowercase{and} A.~Zamani]
{Fuad Kittaneh$^{1}$ \MakeLowercase{and} Ali Zamani$^{2}$}

\address{$^1$Department of Mathematics, The University of Jordan, Amman, Jordan}
\email{fkitt@ju.edu.jo}

\address{$^2$
Department of Mathematics, Farhangian University, Tehran, Iran
\&
School of Mathematics and Computer Sciences, Damghan University, P.O.BOX 36715-364, Damghan, Iran}
\email{zamani.ali85@yahoo.com}


\begin{abstract}
We define a function on the $C^{\ast}$-algebra of all bounded linear Hilbert space operators,
which generalizes the operator radii, and we present some basic properties of this function.
Our results extend several results in the literature.
\end{abstract} \maketitle
\section{Introduction}
Let $\mathbb{B}(\mathcal{H})$ denote the $C^{\ast}$-algebra of all bounded
linear operators on a complex Hilbert space $\big(\mathcal{H}, \langle \cdot, \cdot\rangle \big)$
and $I$ stand for the identity operator on $\mathcal{H}$.
Every operator $X\in \mathbb{B}(\mathcal{H})$ can be represented as $X = {\rm Re}(X) + i{\rm Im}(X)$,
the Cartesian decomposition, where ${\rm Re}(X)= \frac{X+X^*}{2}$ and ${\rm Im}(X)= \frac{X-X^*}{2i}$
are the real and imaginary parts of $X$, respectively.
By $|X|$ we denote the positive square root of $X^*X$, that is, $|X| = (X^*X)^{1/2}$.
Let $X = V|X|$ be the polar decomposition of $X$, where $V$ is some partial isometry.
The polar decomposition satisfies
\begin{align*}
V^*X = |X|, \, V^*V|X| = |X|, \, V^*VX = X, \, X^*=|X|V^*, \, |X^*| = V|X|V^*.
\end{align*}
For $X\in\mathbb{B}(\mathcal{H})$ with a polar decomposition $X = V|X|$, the Aluthge
transform of $X$ is given by $\widetilde{X} = {|X|}^{\frac{1}{2}}V{|X|}^{\frac{1}{2}}$ (see \cite{Aluthge.1966}).
For $\rho>0$ an operator $X\in\mathbb{B}(\mathcal{H})$ is
called a $\rho$-contraction (see \cite{Nagy.Foias.1966}) if there is a Hilbert space $\mathcal{K}(\supseteq \mathcal{H})$
and a unitary operator $U$ on $\mathcal{K}$
such that $X^nx = \rho PU^nx$ for all $x\in\mathcal{H}, n= 1, 2, \cdots$,
where $P$ is the orthogonal projection from $\mathcal{K}$ to $\mathcal{H}$.
Holbrook \cite{Holbrook.1968} and Williams \cite{Williams.1968} defined the operator
radii $w_{\rho}(\cdot)$ as the generalized Minkowski distance functionals
on $\mathbb{B}(\mathcal{H})$, i.e.,
\begin{align*}
w_{\rho}(X)=\inf\left\{t>0: \,t^{-1}X \,\,\mbox{is a $\rho$-contraction}\right\}.
\end{align*}
The operator radius $w_{\rho}(\cdot)$ plays a very important role in the study of unitary
$\rho$-dilations (see, e.g., \cite{Nagy.Foias.Bercovici.Kerchy.2010}).
It is well known that $w_{\rho}(U^*XU) =w_{\rho}(X)$ for all $X$ and all unitary $U\in\mathbb{B}(\mathcal{H})$,
i.e., $w_{\rho}(\cdot)$ is weakly unitarily invariant.
Moreover, $\rho$-radii have the properties:
\begin{align*}
w_1(X) = \|X\|,
\end{align*}
where $\|\!\cdot\!\|$ is the Hilbert space operator norm, that is, $\|X\|= \displaystyle{\sup_{{\|z\|=1}}}\|Xz\|$ and
\begin{align*}
w_2(X) = w(X),
\end{align*}
where $w(\cdot)$ is the numerical radius, that is,
$w(X) = \displaystyle{\sup_{{\|z\|=1}}}|\langle Xz, z\rangle|$.
An important and useful identity for the numerical radius (see \cite[p.~372]{Haagerup.Harpe.1992}) is as follows:
\begin{align}\label{I.12.3}
w(X) = \displaystyle{\sup_{\theta \in \mathbb{R}}}\big\|{\rm Re}(e^{i\theta}X)\big\|.
\end{align}
A very lucid account of the numerical ranges of Hilbert
space operators with applications and complete bibliographical information can be
found in the books of Gustafson--Rao \cite{G.R}, Wu--Gau \cite{W.G} and
Bhunia--Dragomir--Moslehian--Paul \cite{Bhunia.Dragomir.Moslehian.Paul}.

For every $X\in\mathbb{B}(\mathcal{H})$, we also have
\begin{align}\label{I.12.4}
\frac{1}{\rho}\|X\|\leq w_{\rho}(X)\leq \|X\|.
\end{align}
If $X$ is normal, then
$w_{\rho}(X)=\begin{cases}
(2\rho^{-1}-1)\|X\| &\text{if\, $0<\rho<1$}\\
\|X\| &\text{if\, $\rho\geq1$}
\end{cases}$
and if $X$ is $2$-nilpotent, then $w_{\rho}(X)=\frac{1}{\rho}\|X\|$.
Further, $\displaystyle{\lim_{\rho\rightarrow \infty}}w_{\rho}(X)=r(X)$,
where $r(X)$ is the spectral radius of $X\in\mathbb{B}(\mathcal{H})$.
Notice that there is a major difference between
the cases when $0 < \rho \leq 2$ and $2 < \rho < \infty$.
It is known that for $\rho\in (0, 2]$, $w_{\rho}(\cdot)$ is a norm on $\mathbb{B}(\mathcal{H})$
but for $\rho\in (2, \infty)$ is only a quasi-norm.
For proofs and more facts about the operator radii, we refer the reader to
\cite{Ando.Li.2010, Holbrook.1968, Holbrook.1971, K.Z.LAA.2024, Okubo.Ando.1975, Okubo.Ando.1976, Nagy.Foias.Bercovici.Kerchy.2010, Williams.1968, Wu.JPT.1997}.

In this paper, inspired by \cite{Abu-Omar.Kittaneh.LAA2019, Conde.Sababheh.Moradi.2022, M.Z.BJMA.2023, S.K.S.AFA.2022, Z.LAA.2023, Z.W.LAA.2021},
we define a function on $\mathbb{B}(\mathcal{H})$, which generalizes the operator
radii $w_{\rho}(\cdot)$, and we present some basic properties of this function.
Our results extend results in \cite{Abu-Omar.Kittaneh.RMJM.2015, Kittaneh.2003, Kittaneh.2005, Yamazaki.2007}.
\section{Results}
We begin our work with the following definition. Recall that a general $2\times2$ operator matrix in
$\mathbb{B}(\mathcal{H}\oplus\mathcal{H})$ is an operator of the form
$\begin{bmatrix}
S & X \\
Y & T
\end{bmatrix}$, where $X, Y, S, T\in\mathbb{B}(\mathcal{H})$.
\begin{definition}\label{D.1.2}
Let $\rho\in (0, 2]$ and $\nu \in [0, 1]$. The function $\Delta_{_{(\rho,\nu)}}(\cdot)\colon \mathbb{B}(\mathcal{H}) \rightarrow [0,+\infty)$ is defined as
\begin{align*}
\Delta_{_{(\rho,\nu)}}(X)= w\left(\begin{bmatrix}
0 & \alpha X \\
\alpha(1-2\nu)X^*& \beta\left(X+(1-2\nu)X^*\right)
\end{bmatrix}\right),
\end{align*}
where $\alpha=\sqrt{8{\rho}^{-1}-4}$ and $\beta=2{\rho}^{-1}-2$.
\end{definition}
\begin{remark}\label{R.4.42}
From now on, $\rho\in (0, 2]$, $\nu \in [0, 1]$, $\alpha=\sqrt{8{\rho}^{-1}-4}$ and $\beta=2{\rho}^{-1}-2$, unless stated otherwise.
\end{remark}
\begin{remark}\label{R.4.2}
Obviously,
\begin{align*}
\Delta_{_{(\rho,\frac{1}{2})}}(X)= w\left(\begin{bmatrix}
0 & \alpha X \\
0 & \beta X
\end{bmatrix}\right)=
\frac{2}{\rho}\,w\left(\begin{bmatrix}
0 & \sqrt{\rho(2-\rho)}X \\
0 & (1-\rho)X
\end{bmatrix}\right),
\end{align*}
and by \cite[Theorem~3.1]{K.Z.LAA.2024} it follows that
\begin{align}\label{I.13.1}
\Delta_{_{(\rho,\frac{1}{2})}}(X)=w_{\rho}(X).
\end{align}
In particular,
\begin{align*}
\Delta_{_{(1,\frac{1}{2})}}(X)= \|X\| \quad \mbox{and} \quad \Delta_{_{(2,\frac{1}{2})}}(X)= w(X).
\end{align*}
Hence, $\Delta_{_{(\rho,\nu)}}(\cdot)$ generalizes the operator radii $w_{\rho}(\cdot)$.
From \eqref{I.13.1} and \eqref{I.12.4} we also have
\begin{align*}
\frac{1}{\rho} \|X\|\leq \Delta_{_{(\rho,\frac{1}{2})}}(X)\leq \|X\|.
\end{align*}
\end{remark}
\begin{remark}\label{R.4.3}
Since $\left\|\begin{bmatrix}
0 & A\\
A^*& 0
\end{bmatrix}\right\| = \|A\|$ for any operator $A$, we have
\begin{align*}
\Delta_{_{(1,\nu)}}(X)&= 2\,w\left(\begin{bmatrix}
0 & X \\
(1-2\nu)X^*& 0
\end{bmatrix}\right)
\\& = 2\,\displaystyle{\sup_{\theta \in \mathbb{R}}}
\left\|{\rm Re}\left(e^{i\theta}\begin{bmatrix}
0 & X \\
(1-2\nu)X^*& 0
\end{bmatrix}\right)\right\|
\\& = \displaystyle{\sup_{\theta \in \mathbb{R}}}
\left\|\begin{bmatrix}
0 & e^{i\theta}X+e^{-i\theta}(1-2\nu)X\\
\left(e^{i\theta}X+e^{-i\theta}(1-2\nu)X\right)^*& 0
\end{bmatrix}\right\|
\\& = \displaystyle{\sup_{\theta \in \mathbb{R}}}
\left\|e^{i\theta}X+e^{-i\theta}(1-2\nu)X\right\|
\\& = \|X\|\,\displaystyle{\sup_{\theta \in \mathbb{R}}}
\left|1+e^{-2i\theta}(1-2\nu)\right|,
\end{align*}
and hence
\begin{align}\label{I.13.2}
\Delta_{_{(1,\nu)}}(X)=\left(1+|1-2\nu|\right)\|X\|.
\end{align}
In particular,
\begin{align*}
\Delta_{_{(1,0)}}(X)= \Delta_{_{(1,1)}}(X) = 2\|X\|.
\end{align*}
Since $0\leq \nu \leq 1$, it is clear that \eqref{I.13.2} implies
\begin{align*}
\|X\|\leq \Delta_{_{(1,\nu)}}(X) \leq 2\|X\|.
\end{align*}
\end{remark}
\begin{remark}\label{R.4.4}
Since $w\left(\begin{bmatrix}
0 & 0\\
0& -A
\end{bmatrix}\right) = w(A)$ for any operator $A$, we have
\begin{align}\label{R.4.4.I.1}
\Delta_{_{(2,\nu)}}(X)= w\left(\begin{bmatrix}
0 & 0\\
0 & -(X+(1-2\nu)X^*)
\end{bmatrix}\right)
= w\left(X+(1-2\nu)X^*\right).
\end{align}
In particular,
\begin{align*}
\Delta_{_{(2,0)}}(X)= \Delta_{_{(2,1)}}(X) = 2\|{\rm Re}(X)\|.
\end{align*}
\end{remark}
The following theorem states some basic properties of the function $\Delta_{_{(\rho,\nu)}}(\cdot)$.
\begin{theorem}\label{L.0066}
Let $X, Y, Z\in \mathbb{B}(\mathcal{H})$ be such that $Z$ is positive.
The following properties hold:
\begin{itemize}
\item[(i)] $\Delta_{_{(\rho,\nu)}}(X+Y)\leq \Delta_{_{(\rho,\nu)}}(X)+\Delta_{_{(\rho,\nu)}}(Y)$.
\item[(ii)] $\Delta_{_{(\rho,\nu)}}(tX)=|t|\Delta_{_{(\rho,\nu)}}(X)$ for any $t\in\mathbb{R}$.
\item[(iii)] $\Delta_{_{(\rho,\nu)}}(X)> 0$ if $X\neq 0$.
\item[(iv)] $\Delta_{_{(\rho,\nu)}}(U^*XU)= \Delta_{_{(\rho,\nu)}}(X)$ for any unitary $U$.
\item[(v)] $\Delta_{_{(\rho,\nu)}}(Y^*XY)\leq {\|Y\|}^2\Delta_{_{(\rho,\nu)}}(X)$.
\item[(vi)] $\Delta_{_{(\rho,\nu)}}(Z^sXZ^s)\leq \Delta^s_{_{(\rho,\nu)}}(ZXZ)\Delta^{1-s}_{_{(\rho,\nu)}}(X)$
for any $s\in[0, 1]$.
\item[(vii)] $\displaystyle{\lim_{\rho\rightarrow 0}}\,\rho\,\Delta_{_{(\rho,\nu)}}(X)=2w\left(X+(1-2\nu)X^*\right)$.
\end{itemize}
\end{theorem}
\begin{proof}
Statements (i)--(iii) follow from Definition \ref{D.1.2} and basic properties of the norm $w(\cdot)$.

We will assume that $\mathbf{H} = \begin{bmatrix}
0 & \alpha X\\
\alpha(1-2\nu)X^*& \beta\left(X+(1-2\nu)X^*\right)
\end{bmatrix}$, to simplify notation.

(iv) Let $U$ be unitary. Put $\mathbf{U} = \begin{bmatrix}
U & 0 \\
0 & U
\end{bmatrix}$. It is easy to see that $\mathbf{U}$ is unitary. Since the norm $w(\cdot)$ is weakly unitarily invariant, we have
\begin{align*}
\Delta_{_{(\rho,\nu)}}(U^*XU)&= w\left(\begin{bmatrix}
0 & \alpha U^*XU \\
\alpha(1-2\nu)U^*X^*U& \beta\left(U^*XU+(1-2\nu)U^*X^*U\right)
\end{bmatrix}\right)
\\& = w\left(\begin{bmatrix}
U^* & 0\\
0& U^*
\end{bmatrix}\begin{bmatrix}
0 & \alpha X\\
\alpha(1-2\nu)X^*& \beta\left(X+(1-2\nu)X^*\right)
\end{bmatrix}\begin{bmatrix}
U & 0\\
0& U
\end{bmatrix}\right)
\\& = w\left({\mathbf{U}}^*\mathbf{H}\mathbf{U}\right)
= w\left(\mathbf{H}\right) = \Delta_{_{(\rho,\nu)}}(X).
\end{align*}
(v) Let $\mathbf{Y} = \begin{bmatrix}
Y & 0 \\
0 & Y
\end{bmatrix}$. Then $\|\mathbf{Y}\|=\|Y\|$.
Since ${\rm Re}\left(e^{i\theta}{\mathbf{Y}}^*\mathbf{H}\mathbf{Y}\right)={\mathbf{Y}}^*{\rm Re}\left(e^{i\theta}\mathbf{H}\right)\mathbf{Y}$,
by a similar argument as (iv) and by using \eqref{I.12.3}, we have
\begin{align*}
\Delta_{_{(\rho,\nu)}}(Y^*XY)&= w\left({\mathbf{Y}}^*\mathbf{H}\mathbf{Y}\right)
\\& = \displaystyle{\sup_{\theta \in \mathbb{R}}}
\left\|{\rm Re}\left(e^{i\theta}{\mathbf{Y}}^*\mathbf{H}\mathbf{Y}\right)\right\|
\\& = \displaystyle{\sup_{\theta \in \mathbb{R}}}
\left\|{\mathbf{Y}}^*{\rm Re}\left(e^{i\theta}\mathbf{H}\right)\mathbf{Y}\right\|
\\& \leq {\|\mathbf{Y}\|}^2
\displaystyle{\sup_{\theta \in \mathbb{R}}}
\left\|{\rm Re}\left(e^{i\theta}\mathbf{H}\right)\right\|
={\|\mathbf{Y}\|}^2
w\left(\mathbf{H}\right),
\end{align*}
and hence $\Delta_{_{(\rho,\nu)}}(Y^*XY)\leq {\|Y\|}^2\Delta_{_{(\rho,\nu)}}(X)$.

(vi) Let $\mathbf{Z} = \begin{bmatrix}
Z & 0 \\
0 & Z
\end{bmatrix}$. Then $\mathbf{Z}$ is positive.
Let us recall the following classical norm inequality of Heinz \cite{Heinz.1951}:
\begin{align*}
\left\|B^sAC^s\right\|\leq {\left\|BAC\right\|}^{s}{\|A\|}^{1-s}
\end{align*}
for arbitrary operators $A$, and operators $B, C\geq 0$ and real number $s$, $0\leq s\leq1$.
Utilizing a similar argument as (v), the Heinz inequality and \eqref{I.12.3} we have
\begin{align*}
\Delta_{_{(\rho,\nu)}}(Z^sXZ^s)&=
\displaystyle{\sup_{\theta \in \mathbb{R}}}
\left\|{\mathbf{Z}}^{s}{\rm Re}\left(e^{i\theta}\mathbf{H}\right){\mathbf{Z}}^{s}\right\|
\\& \leq \displaystyle{\sup_{\theta \in \mathbb{R}}}
\left({\left\|\mathbf{Z}{\rm Re}\left(e^{i\theta}\mathbf{H}\right)\mathbf{Z}\right\|}^{s}
{\left\|{\rm Re}\left(e^{i\theta}\mathbf{H}\right)\right\|}^{1-s}\right)
\\& =\displaystyle{\sup_{\theta \in \mathbb{R}}}
\left({\left\|{\rm Re}\left(e^{i\theta}\mathbf{Z}\mathbf{H}\mathbf{Z}\right)\right\|}^{s}
{\left\|{\rm Re}\left(e^{i\theta}\mathbf{H}\right)\right\|}^{1-s}\right)
\\& \leq \left(\displaystyle{\sup_{\theta \in \mathbb{R}}}
{\left\|{\rm Re}\left(e^{i\theta}\mathbf{Z}\mathbf{H}\mathbf{Z}\right)\right\|}^{s}\right)
\left(\displaystyle{\sup_{\theta \in \mathbb{R}}}{\left\|{\rm Re}\left(e^{i\theta}\mathbf{H}\right)\right\|}^{1-s}\right)
\\& =w^{s}\left(\mathbf{Z}\mathbf{H}\mathbf{Z}\right)w^{1-s}(\mathbf{H})=\Delta^s_{_{(\rho,\nu)}}(ZXZ)\Delta^{1-s}_{_{(\rho,\nu)}}(X).
\end{align*}

(vii) Let $0<\rho<2$ and let $\mathbf{U} = \begin{bmatrix}
I & 0 \\
0 & -I
\end{bmatrix}$. It is easy to see that $\mathbf{U}$ is unitary and
\begin{align}\label{L.0066.I.1}
{\mathbf{U}}^*\begin{bmatrix}
0 & \alpha X \\
\alpha(1-2\nu)X^*& -\beta\left(X+(1-2\nu)X^*\right)
\end{bmatrix}\mathbf{U} = - \mathbf{H}.
\end{align}
Since the norm $w(\cdot)$ is weakly unitarily invariant, by Definition \ref{D.1.2} and \eqref{L.0066.I.1} we have
\begingroup\makeatletter\def\f@size{10}\check@mathfonts
\begin{align*}
&(2-\rho)\,\Delta_{_{((2-\rho),\nu)}}(X)
\\& \quad = (2-\rho)\,w\left(\begin{bmatrix}
0 & \sqrt{8{(2-\rho)}^{-1}-4}\,X \\
\sqrt{8{(2-\rho)}^{-1}-4}\,(1-2\nu)X^*& \left(2{(2-\rho)}^{-1}-2\right)\left(X+(1-2\nu)X^*\right)
\end{bmatrix}\right)
\\& \quad = (2-\rho)\,w\left(\begin{bmatrix}
0 & \sqrt{4\rho{(2-\rho)}^{-1}}\,X \\
\sqrt{4\rho{(2-\rho)}^{-1}}\,(1-2\nu)X^*& \left(2\rho-2\right){(2-\rho)}^{-1}\left(X+(1-2\nu)X^*\right)
\end{bmatrix}\right)
\\& \quad = w\left(\begin{bmatrix}
0 & \sqrt{4\rho(2-\rho)}\,X \\
\sqrt{4\rho(2-\rho)}\,(1-2\nu)X^*& \left(2\rho-2\right)\left(X+(1-2\nu)X^*\right)
\end{bmatrix}\right)
\\& \quad = \rho\,w\left(\begin{bmatrix}
0 & \sqrt{8{\rho}^{-1}-4}\,X \\
\sqrt{8{\rho}^{-1}-4}\,(1-2\nu)X^*& -\left(2{\rho}^{-1}-2\right)\left(X+(1-2\nu)X^*\right)
\end{bmatrix}\right)
\\& \quad = \rho\,w\left({\mathbf{U}}^*\begin{bmatrix}
0 & \alpha X \\
\alpha(1-2\nu)X^*& -\beta\left(X+(1-2\nu)X^*\right)
\end{bmatrix}\mathbf{U}\right)
\\& \quad = \rho \,w\left(- \mathbf{H}\right) = \rho \,w\left(\mathbf{H}\right)= \rho\,\Delta_{_{(\rho,\nu)}}(X),
\end{align*}
\endgroup
and so $(2-\rho)\,\Delta_{_{((2-\rho),\nu)}}(X) = \rho\,\Delta_{_{(\rho,\nu)}}(X)$.
From this and \eqref{R.4.4.I.1} it follows that
\begin{align*}
\displaystyle{\lim_{\rho\rightarrow 0}}\,\rho\,\Delta_{_{(\rho,\nu)}}(X)=2\Delta_{_{(2,\nu)}}(X)=2w\left(X+(1-2\nu)X^*\right).
\end{align*}
\end{proof}
In the following theorem we give an upper bound for $\Delta_{_{(\rho,\nu)}}(\cdot)$.
\begin{theorem}\label{T.0001}
Let $X\in \mathbb{B}(\mathcal{H})$. Then
\begin{align*}
\Delta_{_{(\rho,\nu)}}(X)
\leq \frac{1+|1-2\nu|}{2}\left(\left\|\begin{bmatrix}
0 & \alpha X \\
0& \beta X
\end{bmatrix}\right\|+w\left(\widetilde{\begin{bmatrix}
0 & \alpha X \\
0& \beta X
\end{bmatrix}}\right)\right).
\end{align*}
\end{theorem}
\begin{proof}
Put $\mathbf{H} = \begin{bmatrix}
0 & \alpha X\\
\alpha(1-2\nu)X^*& \beta\left(X+(1-2\nu)X^*\right)
\end{bmatrix}$ and $\mathbf{G}=\begin{bmatrix}
0 & \alpha X \\
0& \beta X
\end{bmatrix}$, to simplify the writing.
Let $\mathbf{G} = \mathbf{V}|\mathbf{G}|$ be the polar decomposition of $\mathbf{G}$ and let $\theta \in \mathbb{R}$.
Assume now that $z\in \mathcal{H}\oplus\mathcal{H}$ with $\|z\|=1$. We have
\begingroup\makeatletter\def\f@size{9}\check@mathfonts
\begin{align}\label{T.0001.I.101}
&\left|\left\langle{\rm Re}\left(e^{i\theta}\mathbf{H}\right)z, z\right\rangle\right|\nonumber
= \left|{\rm Re}{\left\langle\left(e^{i\theta}\mathbf{H}\right)z, z\right\rangle}\right|\nonumber
\\& \qquad= \left|{\rm Re}{\left\langle\left(e^{i\theta}\mathbf{G}
+(1-2\nu)e^{i\theta}{\mathbf{G}}^*\right)z, z\right\rangle}\right|\nonumber
\\& \qquad= \Big|{\rm Re}{\langle e^{i\theta}\mathbf{G}z, z\rangle}
+(1-2\nu){\rm Re}{\langle e^{i\theta}{\mathbf{G}}^*z, z\rangle}\Big|\nonumber
\\& \qquad= \Big|{\rm Re}{\langle e^{i\theta}\mathbf{V}|\mathbf{G}|z, z\rangle}
+(1-2\nu){\rm Re}{\langle e^{i\theta}|\mathbf{G}|{\mathbf{V}}^*z, z\rangle}\Big|\nonumber
\\& \qquad= \Big|{\rm Re}{\langle |\mathbf{G}|e^{i\theta}z, {\mathbf{V}}^*z\rangle}
+(1-2\nu){\rm Re}\overline{\langle e^{i\theta}|\mathbf{G}|{\mathbf{V}}^*z, z\rangle}\Big|\nonumber
\\& \qquad= \Big|{\rm Re}{\langle |\mathbf{G}|e^{i\theta}z, {\mathbf{V}}^*z\rangle}
+(1-2\nu){\rm Re}{\langle z, e^{i\theta}|\mathbf{G}|{\mathbf{V}}^*z\rangle}\Big|\nonumber
\\& \qquad= \Big|{\rm Re}{\langle |\mathbf{G}|e^{i\theta}z, {\mathbf{V}}^*z\rangle}
+(1-2\nu){\rm Re}{\langle |\mathbf{G}|e^{-i\theta}z, {\mathbf{V}}^*z\rangle}\Big|\nonumber
\\& \qquad\leq \Big|{\rm Re}{\langle |\mathbf{G}|e^{i\theta}z, {\mathbf{V}}^*z\rangle}\Big|
+|1-2\nu|\Big|{\rm Re}{\langle |\mathbf{G}|e^{-i\theta}z, {\mathbf{V}}^*z\rangle}\Big|\nonumber
\\& \qquad= \frac{1}{4}\Big|\Big\langle|\mathbf{G}|\big(e^{i\theta}z+{\mathbf{V}}^*z\big), \big(e^{i\theta}z+{\mathbf{V}}^*z\big)\Big\rangle
- \Big\langle|\mathbf{G}|\big(e^{i\theta}z-{\mathbf{V}}^*z\big), \big(e^{i\theta}z-{\mathbf{V}}^*z\big)\Big\rangle\Big|\nonumber
\\& \qquad \quad +\frac{|1-2\nu|}{4}\Big|\Big\langle|\mathbf{G}|\big(e^{-i\theta}z+{\mathbf{V}}^*z\big), \big(e^{-i\theta}z+{\mathbf{V}}^*z\big)\Big\rangle
- \Big\langle|\mathbf{G}|\big(e^{-i\theta}z-{\mathbf{V}}^*z\big), \big(e^{-i\theta}z-{\mathbf{V}}^*z\big)\Big\rangle\Big|\nonumber
\\& \qquad= \frac{1}{4}\Big|\Big\langle\big(e^{i\theta}I+{\mathbf{V}}^*\big)^*|\mathbf{G}|\big(e^{i\theta}I+{\mathbf{V}}^*\big)z, z\Big\rangle
- \Big\langle\big(e^{i\theta}I-{\mathbf{V}}^*\big)^*|\mathbf{G}|\big(e^{i\theta}I-{\mathbf{V}}^*\big)z, z\Big\rangle\Big|\nonumber
\\& \qquad \quad +\frac{|1-2\nu|}{4}\Big|\Big\langle\big(e^{-i\theta}I+{\mathbf{V}}^*\big)^*|\mathbf{G}|\big(e^{-i\theta}I+{\mathbf{V}}^*\big)z, z\Big\rangle
- \Big\langle\big(e^{-i\theta}I-{\mathbf{V}}^*\big)^*|\mathbf{G}|\big(e^{-i\theta}I-{\mathbf{V}}^*\big)z, z\Big\rangle\Big|\nonumber
\\& \qquad\leq \frac{1}{4}\max\Big\{\Big\langle\big(e^{i\theta}I+{\mathbf{V}}^*\big)^*|\mathbf{G}|\big(e^{i\theta}I+{\mathbf{V}}^*\big)z, z\Big\rangle
, \Big\langle\big(e^{i\theta}I-{\mathbf{V}}^*\big)^*|\mathbf{G}|\big(e^{i\theta}I-{\mathbf{V}}^*\big)z, z\Big\rangle\Big\}\nonumber
\\& \qquad \quad
+\frac{|1-2\nu|}{4}\max\Big\{\Big\langle\big(e^{-i\theta}I+{\mathbf{V}}^*\big)^*|\mathbf{G}|\big(e^{-i\theta}I+{\mathbf{V}}^*\big)z, z\Big\rangle
, \Big\langle\big(e^{-i\theta}I-{\mathbf{V}}^*\big)^*|\mathbf{G}|\big(e^{-i\theta}I-{\mathbf{V}}^*\big)z, z\Big\rangle\Big\}\nonumber
\\& \qquad \qquad \qquad \qquad \Big(\mbox{since $|a-b|\leq \max\{a, b\}$ for any nonnegative numbers $a, b$}\Big)\nonumber
\\& \qquad\leq \frac{1}{4}\max\Big\{\Big\|\big(e^{i\theta}I+{\mathbf{V}}^*\big)^*|\mathbf{G}|\big(e^{i\theta}I+{\mathbf{V}}^*\big)\Big\|
, \Big\|\big(e^{i\theta}I-{\mathbf{V}}^*\big)^*|\mathbf{G}|\big(e^{i\theta}I-{\mathbf{V}}^*\big)\Big\|\Big\}\nonumber
\\& \qquad \quad +\frac{|1-2\nu|}{4}\max\Big\{\Big\|\big(e^{-i\theta}I+{\mathbf{V}}^*\big)^*|\mathbf{G}|\big(e^{-i\theta}I+{\mathbf{V}}^*\big)\Big\|
, \Big\|\big(e^{-i\theta}I-{\mathbf{V}}^*\big)^*|\mathbf{G}|\big(e^{-i\theta}I-{\mathbf{V}}^*\big)\Big\|\Big\}.
\end{align}
\endgroup
Since
\begin{align*}
\Big\|{\mathbf{K}}^*|\mathbf{G}|\mathbf{K}\Big\|
=\Big\|({|\mathbf{G}|}^{\frac{1}{2}}\mathbf{K})^*{|\mathbf{G}|}^{\frac{1}{2}}\mathbf{K}\Big\|
=\Big\|{|\mathbf{G}|}^{\frac{1}{2}}\mathbf{K}({|\mathbf{G}|}^{\frac{1}{2}}\mathbf{K})^*\Big\|
=\Big\|{|\mathbf{G}|}^{\frac{1}{2}}|{\mathbf{K}}^*|^2{|\mathbf{G}|}^{\frac{1}{2}}\Big\|
\end{align*}
for any operator $\mathbf{K}$, by \eqref{T.0001.I.101} it follows that
\begingroup\makeatletter\def\f@size{10}\check@mathfonts
\begin{align}\label{T.0001.I.1}
&\left|\left\langle{\rm Re}\left(e^{i\theta}\mathbf{H}\right)z, z\right\rangle\right|\nonumber
\leq \frac{1}{4}\max\Big\{\Big\|{|\mathbf{G}|}^{\frac{1}{2}}{\big|e^{-i\theta}I+\mathbf{V}\big|}^2{|\mathbf{G}|}^{\frac{1}{2}}\Big\|
, \Big\|{|\mathbf{G}|}^{\frac{1}{2}}{\big|e^{-i\theta}I-\mathbf{V}\big|}^2{|\mathbf{G}|}^{\frac{1}{2}}\Big\|\Big\}\nonumber
\\& \qquad \qquad\qquad \quad +\frac{|1-2\nu|}{4}\max\Big\{\Big\|{|\mathbf{G}|}^{\frac{1}{2}}{\big|e^{i\theta}I+\mathbf{V}\big|}^2{|\mathbf{G}|}^{\frac{1}{2}}\Big\|
, \Big\|{|\mathbf{G}|}^{\frac{1}{2}}{\big|e^{i\theta}I-\mathbf{V}\big|}^2{|\mathbf{G}|}^{\frac{1}{2}}\Big\|\Big\}.
\end{align}
\endgroup
We have
\begingroup\makeatletter\def\f@size{10}\check@mathfonts
\begin{align*}
{|\mathbf{G}|}^{\frac{1}{2}}{\big|e^{-i\theta}I+\mathbf{V}\big|}^2{|\mathbf{G}|}^{\frac{1}{2}}
&= |\mathbf{G}|+e^{i\theta}{|\mathbf{G}|}^{\frac{1}{2}}\mathbf{V}{|\mathbf{G}|}^{\frac{1}{2}}
+e^{-i\theta}{|\mathbf{G}|}^{\frac{1}{2}}{\mathbf{V}}^*{|\mathbf{G}|}^{\frac{1}{2}}
+{|\mathbf{G}|}^{\frac{1}{2}}{\mathbf{V}}^*\mathbf{V}{|\mathbf{G}|}^{\frac{1}{2}}
\\& = |\mathbf{G}|+e^{i\theta}\widetilde{\mathbf{G}}
+e^{-i\theta}\big(\widetilde{\mathbf{G}}\big)^{*}+ |\mathbf{G}|,
\end{align*}
\endgroup
and so
\begin{align}\label{T.0001.I.2}
{|\mathbf{G}|}^{\frac{1}{2}}{\big|e^{-i\theta}I+\mathbf{V}\big|}^2{|\mathbf{G}|}^{\frac{1}{2}}
= 2|\mathbf{G}|+2{\rm Re}\big(e^{i\theta}\widetilde{\mathbf{G}}\big).
\end{align}
Similarly,
\begin{align}\label{T.0001.I.3}
{|\mathbf{G}|}^{\frac{1}{2}}{\big|e^{-i\theta}I-\mathbf{V}\big|}^2{|\mathbf{G}|}^{\frac{1}{2}}
= 2|\mathbf{G}|-2{\rm Re}\big(e^{i\theta}\widetilde{\mathbf{G}}\big),
\end{align}
\begin{align}\label{T.0001.I.4}
{|\mathbf{G}|}^{\frac{1}{2}}{\big|e^{i\theta}I+\mathbf{V}\big|}^2{|\mathbf{G}|}^{\frac{1}{2}}
= 2|\mathbf{G}|+2{\rm Re}\big(e^{-i\theta}\widetilde{\mathbf{G}}\big),
\end{align}
and
\begin{align}\label{T.0001.I.5}
{|\mathbf{G}|}^{\frac{1}{2}}{\big|e^{i\theta}I-\mathbf{V}\big|}^2{|\mathbf{G}|}^{\frac{1}{2}}
= 2|\mathbf{G}|-2{\rm Re}\big(e^{-i\theta}\widetilde{\mathbf{G}}\big).
\end{align}
Now from \eqref{T.0001.I.1}--\eqref{T.0001.I.5} it follows that
\begingroup\makeatletter\def\f@size{10}\check@mathfonts
\begin{align*}
&\left|\left\langle{\rm Re}\left(e^{i\theta}\mathbf{H}\right)z, z\right\rangle\right|
\\& \qquad\leq \frac{1}{2}\max\Big\{\Big\||\mathbf{G}|+{\rm Re}\big(e^{i\theta}\widetilde{\mathbf{G}}\big)\Big\|
, \Big\||\mathbf{G}|-{\rm Re}\big(e^{i\theta}\widetilde{\mathbf{G}}\big)\Big\|\Big\}
\\& \qquad \quad +\frac{|1-2\nu|}{2}\max\Big\{\Big\||\mathbf{G}|+{\rm Re}\big(e^{-i\theta}\widetilde{\mathbf{G}}\big)\Big\|
, \Big\||\mathbf{G}|-{\rm Re}\big(e^{-i\theta}\widetilde{\mathbf{G}}\big)\Big\|\Big\}
\\& \qquad\leq \frac{1}{2}\Big(\big\||\mathbf{G}|\big\|+\big\|{\rm Re}\big(e^{i\theta}\widetilde{\mathbf{G}}\big)\big\|\Big)
+\frac{|1-2\nu|}{2}\Big(\big\||\mathbf{G}|\big\|+\big\|{\rm Re}\big(e^{-i\theta}\widetilde{\mathbf{G}}\big)\big\|\Big)
\\& \qquad\leq \frac{1}{2}\Big(\big\|\mathbf{G}\big\|+w\big(\widetilde{\mathbf{G}}\big)\Big)
+\frac{|1-2\nu|}{2}\Big(\big\|\mathbf{G}\big\|+w\big(\widetilde{\mathbf{G}}\big)\Big),
\end{align*}
\endgroup
and hence
\begingroup\makeatletter\def\f@size{10}\check@mathfonts
\begin{align}\label{T.0001.I.6}
\left|\left\langle{\rm Re}\left(e^{i\theta}\mathbf{H}\right)z, z\right\rangle\right|
\leq \frac{1+|1-2\nu|}{2}\Big(\big\|\mathbf{G}\big\|+w\big(\widetilde{\mathbf{G}}\big)\Big).
\end{align}
\endgroup
Taking the supremum over $z\in \mathcal{H}\oplus\mathcal{H}$ with $\|z\|=1$ in \eqref{T.0001.I.6}, we get
\begingroup\makeatletter\def\f@size{10}\check@mathfonts
\begin{align}\label{T.0001.I.7}
\left\|{\rm Re}\left(e^{i\theta}\mathbf{H}\right)\right\|
\leq \frac{1+|1-2\nu|}{2}\Big(\big\|\mathbf{G}\big\|+w\big(\widetilde{\mathbf{G}}\big)\Big).
\end{align}
\endgroup
Finally, taking the supremum over $\theta\in\mathbb{R}$ in \eqref{T.0001.I.7}, and by using \eqref{I.12.3}, we deduce that
\begin{align*}
w\left(\mathbf{H}\right)
\leq \frac{1+|1-2\nu|}{2}\Big(\big\|\mathbf{G}\big\|+w\big(\widetilde{\mathbf{G}}\big)\Big),
\end{align*}
and therefore
\begin{align*}
\Delta_{_{(\rho,\nu)}}(X)
\leq \frac{1+|1-2\nu|}{2}\Big(\big\|\mathbf{G}\big\|+w\big(\widetilde{\mathbf{G}}\big)\Big).
\end{align*}
\end{proof}
As a consequence of Theorem \ref{T.0001}, we have the following result.
\begin{corollary}\label{C.0002}
Let $X\in \mathbb{B}(\mathcal{H})$. Then
\begin{align}\label{C.0002.I.001}
w\left(X+(1-2\nu)X^*\right)
\leq \frac{1+|1-2\nu|}{2}\left(\|X\|+w(\widetilde{X})\right).
\end{align}
In particular,
\begin{align}\label{C.0002.I.002}
w(X)\leq \frac{1}{2}\left(\|X\|+w(\widetilde{X})\right).
\end{align}
\end{corollary}
\begin{proof}
The inequality \eqref{C.0002.I.001} follows from Theorem \ref{T.0001} by letting $\rho= 2$ and \eqref{R.4.4.I.1}.
The inequality \eqref{C.0002.I.002} also follows from \eqref{C.0002.I.001} by letting $\nu=\frac{1}{2}$.
\end{proof}
\begin{remark}\label{R.00021}
The inequality \eqref{C.0002.I.002} in Corollary \ref{C.0002} is due to Yamazaki \cite[Theorem~2.1]{Yamazaki.2007}.
\end{remark}
As another consequence of Theorem \ref{T.0001}, we have the following result.
\begin{corollary}\label{C.000209}
Let $X\in \mathbb{B}(\mathcal{H})$. Then
\begin{align*}
\Delta_{_{(\rho,\nu)}}(X)
\leq \frac{1+|1-2\nu|}{2}\left(\left\|\begin{bmatrix}
0 & \alpha X \\
0& \beta X
\end{bmatrix}\right\|+\sqrt{|\beta|}\,{\left\|{\begin{bmatrix}
0 & \alpha X^2 \\
0& \beta X^2
\end{bmatrix}}\right\|}^{\frac{1}{2}}\right).
\end{align*}
\end{corollary}
\begin{proof}
Let $\mathbf{G} = \mathbf{V}|\mathbf{G}|$ be the polar decomposition of $\mathbf{G}=\begin{bmatrix}
0 & \alpha X \\
0& \beta X
\end{bmatrix}$. It is easy to check that ${\mathbf{G}}^2=\beta \begin{bmatrix}
0 & \alpha X^2 \\
0& \beta X^2
\end{bmatrix}$.
By the Heinz inequality it follows that
\begin{align*}
w\big(\widetilde{\mathbf{G}}\big)\leq \|\widetilde{\mathbf{G}}\|
= \left\|{|\mathbf{G}|}^{\frac{1}{2}}\mathbf{V}{|\mathbf{G}|}^{\frac{1}{2}}\right\|
\leq {\left\||\mathbf{G}|\mathbf{V}|\mathbf{G}|\right\|}^{\frac{1}{2}}{\left\|\mathbf{V}\right\|}^{\frac{1}{2}},
={\left\|{\mathbf{G}}^2\right\|}^{\frac{1}{2}}
\end{align*}
and so by Theorem \ref{T.0001} we obtain
\begin{align*}
\Delta_{_{(\rho,\nu)}}(X)
\leq \frac{1+|1-2\nu|}{2}\Big(\big\|\mathbf{G}\big\|+{\left\|{\mathbf{G}}^2\right\|}^{\frac{1}{2}}\Big),
\end{align*}
as required.
\end{proof}
\begin{remark}\label{R.000209}
Kittaneh \cite[Theorem~1]{Kittaneh.2003} proved that $w(X)\leq \frac{1}{2}\left(\|X\|+{\|X^2\|}^{\frac{1}{2}}\right)$,
which easily follows from Corollary \ref{C.000209} by letting $\rho= 2$, $\nu=\frac{1}{2}$ and \eqref{R.4.4.I.1}.
\end{remark}
In the following theorem we state another upper bound for $\Delta_{_{(\rho,\nu)}}(\cdot)$.
\begin{theorem}\label{T.000263}
Let $X\in \mathbb{B}(\mathcal{H})$ and $\lambda\in \mathbb{C}\setminus\{0\}$. Then
\begin{align*}
\Delta^2_{_{(\rho,\nu)}}(X)
&\leq (1-2\nu)^2w^2_{\rho}(X)
+ 2\left|\frac{(1-2\nu)\beta}{\lambda}\right|w_{\rho}(X^2)
\\& \qquad +\frac{2|1-2\nu|\max\big\{1, |\lambda - 1|\big\}+|\lambda|}{2|\lambda|}\,
\left\|\begin{bmatrix}
\mathbf{A} & \mathbf{B}\\
\mathbf{B}& \mathbf{C}
\end{bmatrix}\right\|,
\end{align*}
where $\mathbf{A}={\alpha}^2|X^*|^2$, $\mathbf{B}=\alpha\beta|X^*|^2$ and
$\mathbf{C}=({\alpha}^2+{\beta}^2)|X|^2+{\beta}^2|X^*|^2$.
\end{theorem}
\begin{proof}
Again we will assume that $\mathbf{H} = \begin{bmatrix}
0 & \alpha X\\
\alpha(1-2\nu)X^*& \beta\left(X+(1-2\nu)X^*\right)
\end{bmatrix}$ and $\mathbf{G}=\begin{bmatrix}
0 & \alpha X \\
0& \beta X
\end{bmatrix}$, to simplify notations.
Then simple computation shows that
\begin{align}\label{T.000263.I.0}
{\mathbf{G}}^2=\beta\begin{bmatrix}
0 & \alpha X^2 \\
0& \beta X^2
\end{bmatrix} \quad \mbox{and} \quad
|\mathbf{G}|^2+|{\mathbf{G}}^*|^2
=\begin{bmatrix}
\mathbf{A} & \mathbf{B}\\
\mathbf{B}& \mathbf{C}
\end{bmatrix}.
\end{align}
Let us recall the following extension of the Buzano inequality \cite[Lemma~2.2]{K.Z.JCAM.2023}:
\begin{align}\label{T.000263.I.1}
\big|\langle u, e\rangle\langle e, v\rangle\big|
\leq \frac{1}{|\lambda|}\Big(\max\big\{1, |\lambda - 1|\big\}\|u\|\|v\| + |\langle u, v\rangle|\Big)
\end{align}
for vectors $u, v, e \in \mathcal{H}\oplus\mathcal{H}$ with $\|e\| = 1$ and $\lambda\in \mathbb{C}\setminus\{0\}$.
Assume now that $z\in \mathcal{H}\oplus\mathcal{H}$ with $\|z\|=1$.
Since ${\left|\langle \mathbf{G}z, z\rangle\right|}^2\leq \left\langle|\mathbf{G}|z, z\right\rangle\left\langle|{\mathbf{G}}^*|z, z\right\rangle$ (see, e.g., \cite[pp.~75--76]{Halmos.1982}), by \eqref{T.000263.I.1} and the
arithmetic-geometric mean inequality, we have
\begingroup\makeatletter\def\f@size{10}\check@mathfonts
\begin{align}\label{T.000263.I.2}
&{\left|\left\langle\mathbf{H}z, z\right\rangle\right|}^2
= {\left|{\left\langle \mathbf{G}z
+(1-2\nu){\mathbf{G}}^*z, z\right\rangle}\right|}^2\nonumber
\\& \leq \left(\left|\langle \mathbf{G}z, z\rangle\right|+|1-2\nu|\left|\langle{\mathbf{G}}^*z, z\rangle\right|\right)^2\nonumber
\\& = {\left|\langle \mathbf{G}z, z\rangle\right|}^2+(1-2\nu)^2{\left|\langle{\mathbf{G}}^*z, z\rangle\right|}^2
+ 2|1-2\nu|\left|\langle \mathbf{G}z, z\rangle\langle z, {\mathbf{G}}^*z\rangle\right|\nonumber
\\& \leq \left\langle|\mathbf{G}|z, z\right\rangle\left\langle|{\mathbf{G}}^*|z, z\right\rangle
+(1-2\nu)^2{\left|\langle z, \mathbf{G}z\rangle\right|}^2\nonumber
\\& \qquad + \frac{2|1-2\nu|}{|\lambda|}\Big(\max\big\{1, |\lambda - 1|\big\}\|\mathbf{G}z\|\|{\mathbf{G}}^*z\|+\left|\langle\mathbf{G}z, {\mathbf{G}}^*z\rangle\right|\Big)\nonumber
\\&\leq \frac{{\left\langle|\mathbf{G}|z, z\right\rangle}^2+{\left\langle|{\mathbf{G}}^*|z, z\right\rangle}^2}{2}
+ (1-2\nu)^2{\left|\langle\mathbf{G}z, z\rangle\right|}^2\nonumber
\\& \qquad +\frac{|1-2\nu|\max\big\{1, |\lambda - 1|\big\}}{|\lambda|}\big({\|\mathbf{G}z\|}^2+{\|{\mathbf{G}}^*z\|}^2\big)
+ \frac{2|1-2\nu|}{|\lambda|}\left|\langle{\mathbf{G}}^2z, z\rangle\right|\nonumber
\\&= \frac{{\left\langle|\mathbf{G}|z, z\right\rangle}^2+{\left\langle|{\mathbf{G}}^*|z, z\right\rangle}^2}{2}
+ (1-2\nu)^2{\left|\langle\mathbf{G}z, z\rangle\right|}^2\nonumber
\\& \qquad +\frac{|1-2\nu|\max\big\{1, |\lambda - 1|\big\}}{|\lambda|}\big(\left\langle|\mathbf{G}|^2z, z\right\rangle+\left\langle|{\mathbf{G}}^*|^2z, z\right\rangle\big)+ \frac{2|1-2\nu|}{|\lambda|}\left|\langle{\mathbf{G}}^2z, z\rangle\right|.
\end{align}
\endgroup
Further, since ${\langle Px, x\rangle}^2\leq \langle P^2x, x\rangle$ for any positive operator $P$ and
any unit vector $x$ (see, e.g., \cite{Kittaneh.1988}), by \eqref{T.000263.I.2} and \eqref{T.000263.I.0} it follows that
\begingroup\makeatletter\def\f@size{10}\check@mathfonts
\begin{align*}
{\left|\left\langle\mathbf{H}z, z\right\rangle\right|}^2
&\leq \frac{\left\langle|\mathbf{G}|^2z, z\right\rangle+\left\langle|{\mathbf{G}}^*|^2z, z\right\rangle}{2}
+ (1-2\nu)^2{\left|\langle\mathbf{G}z, z\rangle\right|}^2
\\& \qquad +\frac{|1-2\nu|\max\big\{1, |\lambda - 1|\big\}}{|\lambda|}\big(\left\langle|\mathbf{G}|^2z, z\right\rangle+\left\langle|{\mathbf{G}}^*|^2z, z\right\rangle\big)
+ \frac{2|1-2\nu|}{|\lambda|}\left|\langle{\mathbf{G}}^2z, z\rangle\right|
\\&= (1-2\nu)^2{\left|\langle\mathbf{G}z, z\rangle\right|}^2+ \frac{2|1-2\nu|}{|\lambda|}\left|\langle{\mathbf{G}}^2z, z\rangle\right|
\\& \qquad +\frac{2|1-2\nu|\max\big\{1, |\lambda - 1|\big\}+|\lambda|}{2|\lambda|}\big(\left\langle\left(|\mathbf{G}|^2+|{\mathbf{G}}^*|^2\right)z, z\right\rangle\big)
\\&\leq (1-2\nu)^2w^2(\mathbf{G})
+ \frac{2|1-2\nu|}{|\lambda|}w({\mathbf{G}}^2)
\\& \qquad +\frac{2|1-2\nu|\max\big\{1, |\lambda - 1|\big\}+|\lambda|}{2|\lambda|}\,
\left\||\mathbf{G}|^2+|{\mathbf{G}}^*|^2\right\|
\\&=(1-2\nu)^2w^2_{\rho}(X)
+ \frac{2|1-2\nu|}{|\lambda|}|\beta|w_{\rho}(X^2)
\\& \qquad +\frac{2|1-2\nu|\max\big\{1, |\lambda - 1|\big\}+|\lambda|}{2|\lambda|}\,
\left\|\begin{bmatrix}
\mathbf{A} & \mathbf{B}\\
\mathbf{B}& \mathbf{C}
\end{bmatrix}\right\|,
\end{align*}
\endgroup
and hence
\begingroup\makeatletter\def\f@size{10}\check@mathfonts
\begin{align*}
{\left|\left\langle\mathbf{H}z, z\right\rangle\right|}^2
&\leq (1-2\nu)^2w^2_{\rho}(X)
+ 2\left|\frac{(1-2\nu)\beta}{\lambda}\right|w_{\rho}(X^2)
\\& \qquad +\frac{2|1-2\nu|\max\big\{1, |\lambda - 1|\big\}+|\lambda|}{2|\lambda|}\,
\left\|\begin{bmatrix}
\mathbf{A} & \mathbf{B}\\
\mathbf{B}& \mathbf{C}
\end{bmatrix}\right\|.
\end{align*}
\endgroup
Taking the supremum over unit vectors $z\in \mathcal{H}\oplus\mathcal{H}$ in the above inequality, we deduce the desired result.
\end{proof}
Consequences of Theorem \ref{T.000263} can be stated as follows.
\begin{corollary}\label{C.0002.1}
Let $X\in \mathbb{B}(\mathcal{H})$. Then
\begingroup\makeatletter\def\f@size{10}\check@mathfonts
\begin{align*}
\Delta^2_{_{(\rho,\nu)}}(X)
\leq (1-2\nu)^2w^2_{\rho}(X)
+ \left|(1-2\nu)\beta\right|w_{\rho}(X^2)
+\frac{|1-2\nu|+1}{2}\,
\left\|\begin{bmatrix}
\mathbf{A} & \mathbf{B}\\
\mathbf{B}& \mathbf{C}
\end{bmatrix}\right\|,
\end{align*}
\endgroup
where $\mathbf{A}, \mathbf{B}$ and $\mathbf{C}$ are the same as in Theorem \ref{T.000263}.
\end{corollary}
\begin{proof}
The proof follows from Theorem \ref{T.000263} by letting $\lambda=2$.
\end{proof}
\begin{corollary}\label{C.0002.3}
Let $X\in \mathbb{B}(\mathcal{H})$. Then
\begingroup\makeatletter\def\f@size{10}\check@mathfonts
\begin{align}\label{C.0002.3.I.1}
w^2\left(X+(1-2\nu)X^*\right)\leq (1-2\nu)^2w^2(X)+ |1-2\nu|w(X^2)
+\frac{|1-2\nu|+1}{2}\left\||X|^2+|X^*|^2\right\|.
\end{align}
\endgroup
In particular,
\begin{align}\label{C.0002.3.I.2}
w^2(X)\leq \frac{1}{2}\left\||X|^2+|X^*|^2\right\|.
\end{align}
\end{corollary}
\begin{proof}
Let $\mathbf{A}, \mathbf{B}$ and $\mathbf{C}$ be the same as in Theorem \ref{T.000263}.
For $\rho=2$, we have $\mathbf{A}=\mathbf{B}=0$ and $\mathbf{C}=|X|^2+|X^*|^2$. Thus,
\begin{align*}
\left\|\begin{bmatrix}
\mathbf{A} & \mathbf{B}\\
\mathbf{B}& \mathbf{C}
\end{bmatrix}\right\|
= \left\||X|^2+|X^*|^2\right\|.
\end{align*}
Now, by \eqref{R.4.4.I.1} and Corollary \ref{C.0002.1} we deduce the desired result.
The inequality \eqref{C.0002.3.I.2} also follows from \eqref{C.0002.3.I.1} by letting $\nu=\frac{1}{2}$.
\end{proof}
\begin{remark}\label{R.0002121}
The inequality \eqref{C.0002.3.I.2} in Corollary \ref{C.0002.3} is due to Kittaneh \cite[Theorem~1]{Kittaneh.2005}.
\end{remark}
\begin{corollary}\label{C.0002.2}
Let $X\in \mathbb{B}(\mathcal{H})$. Then
\begin{align*}
\Delta^2_{_{(\rho,\nu)}}(X)
\leq (1-2\nu)^2w^2_{\rho}(X)+
\frac{2|1-2\nu|+1}{2}
\left\|\begin{bmatrix}
\mathbf{A} & \mathbf{B}\\
\mathbf{B}& \mathbf{C}
\end{bmatrix}\right\|,
\end{align*}
where $\mathbf{A}, \mathbf{B}$ and $\mathbf{C}$ are the same as in Theorem \ref{T.000263}.
\end{corollary}
\begin{proof}
The proof follows from Theorem \ref{T.000263} by letting $\lambda= n$ and $n\rightarrow\infty$.
\end{proof}
For an operator $X$, its Crawford number $c(X)$ is defined by
$c(X) = \displaystyle{\inf_{{\|z\|=1}}}|\langle Xz, z\rangle|$.
This concept is useful in studying linear operators and has
attracted the attention of many authors in the last few decades
(see, e.g., \cite{Wang.Wu.Gau.LAA.2010} and its references).
\begin{theorem}\label{T.00044}
Let $X\in \mathbb{B}(\mathcal{H})$. Then
\begingroup\makeatletter\def\f@size{10}\check@mathfonts
\begin{align*}
\frac{1}{4}\,\left\|\begin{bmatrix}
\mathbf{M} & \mathbf{N}\\
{\mathbf{N}}^*& \mathbf{P}
\end{bmatrix}\right\|+\frac{1}{2}\,c\left(\begin{bmatrix}
\mathbf{Q} & \mathbf{R}\\
\mathbf{T}& \mathbf{S}
\end{bmatrix}\right) \leq \Delta^2_{_{(\rho,\nu)}}(X)
\leq \frac{1}{4}\,\left\|\begin{bmatrix}
\mathbf{M} & \mathbf{N}\\
{\mathbf{N}}^*& \mathbf{P}
\end{bmatrix}\right\|+\frac{1}{2}\,w\left(\begin{bmatrix}
\mathbf{Q} & \mathbf{R}\\
\mathbf{T}& \mathbf{S}
\end{bmatrix}\right),
\end{align*}
\endgroup
where
\begingroup\makeatletter\def\f@size{10}\check@mathfonts
\begin{align*}
\mathbf{Q}=\alpha^2(1-2\nu)|X^*|^2, \qquad \mathbf{R}=\alpha\beta\left(X^2+(1-2\nu)|X^*|^2\right),
\end{align*}
\endgroup
\begingroup\makeatletter\def\f@size{10}\check@mathfonts
\begin{align*}
\mathbf{T}=\alpha\beta(1-2\nu)\left(|X^*|^2+(1-2\nu)(X^*)^2\right), \qquad \mathbf{S}=\alpha^2(1-2\nu)|X|^2+\beta^2\left(X+(1-2\nu)X^*\right)^2,
\end{align*}
\endgroup
\begingroup\makeatletter\def\f@size{10}\check@mathfonts
\begin{align*}
\mathbf{M}=\alpha^2\left(1+(1-2\nu)^2\right)|X^*|^2, \qquad \mathbf{N}=2\alpha\beta(1-2\nu)X^2 + \alpha\beta\left(1+(1-2\nu)^2\right)|X^*|^2,
\end{align*}
\endgroup
and
\begingroup\makeatletter\def\f@size{10}\check@mathfonts
\begin{align*}
\mathbf{P}=\alpha^2\left(1+(1-2\nu)^2\right)|X|^2+\beta^2
\left(\big|X+(1-2\nu)X^*\big|^2+\big|X^*+(1-2\nu)X\big|^2\right).
\end{align*}
\endgroup
\end{theorem}
\begin{proof}
Again put $\mathbf{H} = \begin{bmatrix}
0 & \alpha X\\
\alpha(1-2\nu)X^*& \beta\left(X+(1-2\nu)X^*\right)
\end{bmatrix}$, to simplify the writing.
A straightforward computation,
whose details we omit, shows that
\begin{align}\label{T.00044.I.1}
{\mathbf{H}}^2=\begin{bmatrix}
\mathbf{Q} & \mathbf{R}\\
\mathbf{T}& \mathbf{S}
\end{bmatrix}
\quad
\mbox{and} \quad
|\mathbf{H}|^2+|{\mathbf{H}}^*|^2=\begin{bmatrix}
\mathbf{M} & \mathbf{N}\\
{\mathbf{N}}^*& \mathbf{P}
\end{bmatrix}.
\end{align}
To prove the first inequality, fix $z\in \mathcal{H}\oplus\mathcal{H}$ with $\|z\|=1$, and choose $\theta\in\mathbb{R}$
so that $\left\langle{\mathbf{H}}^2z, z\right\rangle = e^{i\theta}\left|\left\langle{\mathbf{H}}^2z, z\right\rangle\right|$.
Note that
\begin{align*}
\left({\rm Re}\left(e^{-i\frac{\theta}{2}}\mathbf{H}\right)\right)^2=
\frac{1}{4}\left(|\mathbf{H}|^2+|{\mathbf{H}}^*|^2 + 2{\rm Re}\left(e^{-i\theta}{\mathbf{H}}^2\right)\right)
\end{align*}
and then by \eqref{I.12.3} we have
\begin{align*}
\Delta^2_{_{(\rho,\nu)}}(X)&= w^2(\mathbf{H})
\\& \geq {\left\|{\rm Re}\left(e^{-i\frac{\theta}{2}}\mathbf{H}\right)\right\|}^2
\\& = \left\|\left({\rm Re}\left(e^{-i\frac{\theta}{2}}\mathbf{H}\right)\right)^2\right\|
\\& \geq \frac{1}{4}\left|\left\langle\left(|\mathbf{H}|^2+|{\mathbf{H}}^*|^2 + 2{\rm Re}\left(e^{-i\theta}{\mathbf{H}}^2\right)\right)z, z\right\rangle\right|
\\& = \frac{1}{4}\left|\left\langle\left(|\mathbf{H}|^2+|{\mathbf{H}}^*|^2\right)z, z\right\rangle
+ 2{\rm Re}\left(e^{-i\theta}\langle{{\mathbf{H}}}^2z, z\rangle\right)\right|
\\& = \frac{1}{4}\left\langle\left(|\mathbf{H}|^2+|{\mathbf{H}}^*|^2\right)z, z\right\rangle
+ \frac{1}{2}\left|\left\langle{\mathbf{H}}^2z, z\right\rangle\right|
\\& \geq \frac{1}{4}\left\langle\left(|\mathbf{H}|^2+|{\mathbf{H}}^*|^2\right)z, z\right\rangle
+ \frac{1}{2}c\left({\mathbf{H}}^2\right),
\end{align*}
and so
\begin{align}\label{T.00044.I.2}
\Delta^2_{_{(\rho,\nu)}}(X)\geq \left\langle\left(|\mathbf{H}|^2+|{\mathbf{H}}^*|^2\right)z, z\right\rangle
+ 2c\left({\mathbf{H}}^2\right).
\end{align}
Taking the supremum over unit vectors $z\in \mathcal{H}\oplus\mathcal{H}$ in \eqref{T.00044.I.2} and using \eqref{T.00044.I.1}, we
arrive at
\begin{align}\label{T.00044.I.3}
\frac{1}{4}\,\left\|\begin{bmatrix}
\mathbf{M} & \mathbf{N}\\
{\mathbf{N}}^*& \mathbf{P}
\end{bmatrix}\right\|+\frac{1}{2}\,c\left(\begin{bmatrix}
\mathbf{Q} & \mathbf{R}\\
\mathbf{T}& \mathbf{S}
\end{bmatrix}\right) \leq \Delta^2_{_{(\rho,\nu)}}(X).
\end{align}
To prove the second inequality, again by \eqref{I.12.3} we have
\begin{align*}
\Delta^2_{_{(\rho,\nu)}}(X)&= w^2(\mathbf{H})
\\&= \displaystyle{\sup_{\theta \in \mathbb{R}}}
{\left\|{\rm Re}\left(e^{i\theta}\mathbf{H}\right)\right\|}^2
\\& = \displaystyle{\sup_{\theta \in \mathbb{R}}}\left\|\left({\rm Re}\left(e^{i\theta}\mathbf{H}\right)\right)^2\right\|
\\& = \frac{1}{4}\displaystyle{\sup_{\theta \in \mathbb{R}}}\left\||\mathbf{H}|^2+|{\mathbf{H}}^*|^2 + 2{\rm Re}\left(e^{2i\theta}{\mathbf{H}}^2\right)\right\|
\\& \leq \frac{1}{4}\left\||\mathbf{H}|^2+|{\mathbf{H}}^*|^2\right\|
+\frac{1}{2}\displaystyle{\sup_{\theta \in \mathbb{R}}}\left\|{\rm Re}\left(e^{2i\theta}{\mathbf{H}}^2\right)\right\|
\\& = \frac{1}{4}\left\||\mathbf{H}|^2+|{\mathbf{H}}^*|^2\right\|
+\frac{1}{2}w\left({\mathbf{H}}^2\right),
\end{align*}
and hence by \eqref{T.00044.I.1} we conclude that
\begin{align}\label{T.00044.I.4}
\Delta^2_{_{(\rho,\nu)}}(X)
\leq \frac{1}{4}\,\left\|\begin{bmatrix}
\mathbf{M} & \mathbf{N}\\
{\mathbf{N}}^*& \mathbf{P}
\end{bmatrix}\right\|+\frac{1}{2}\,w\left(\begin{bmatrix}
\mathbf{Q} & \mathbf{R}\\
\mathbf{T}& \mathbf{S}
\end{bmatrix}\right).
\end{align}
Utilizing \eqref{T.00044.I.3} and \eqref{T.00044.I.4}, we deduce the desired result.
\end{proof}
As an application of Theorem \ref{T.00044} we now prove the following result.
\begin{corollary}\label{C.0002.30}
Let $X\in \mathbb{B}(\mathcal{H})$. Then
\begingroup\makeatletter\def\f@size{10}\check@mathfonts
\begin{align}\label{C.00044.I.1}
&\frac{1}{4}\left\|{\big|X+(1-2\nu)X^*\big|}^2+{\big|X^*+(1-2\nu)X\big|}^2\right\|
+\frac{1}{2}c\left({(X+(1-2\nu)X^*)}^2\right)\nonumber
\\&\qquad \leq w^2\left(X+(1-2\nu)X^*\right)\nonumber
\\&\qquad \leq \frac{1}{4}\left\|{\big|X+(1-2\nu)X^*\big|}^2+{\big|X^*+(1-2\nu)X\big|}^2\right\|
+\frac{1}{2}w\left({(X+(1-2\nu)X^*)}^2\right).
\end{align}
\endgroup
In particular,
\begin{align}\label{C.00044.I.2}
\frac{1}{4}\left\|{|X|}^2+{|X^*|}^2\right\|
+\frac{1}{2}c\left({X}^2\right)
\leq w^2(X)\leq \frac{1}{4}\left\|{|X|}^2+{|X^*|}^2\right\|
+\frac{1}{2}w\left({X}^2\right).
\end{align}
\end{corollary}
\begin{proof}
Let $\mathbf{M}, \mathbf{N}, \mathbf{P}, \mathbf{Q}, \mathbf{R}, \mathbf{S}$ and $\mathbf{T}$ be the same as in Theorem \ref{T.00044}.
For $\rho=2$, we have $\mathbf{M}=\mathbf{N}=\mathbf{Q}=\mathbf{R}=\mathbf{T}=0$,
$\mathbf{P}={\big|X+(1-2\nu)X^*\big|}^2+{\big|X^*+(1-2\nu)X\big|}^2$ and $\mathbf{S}={(X+(1-2\nu)X^*)}^2$. Hence,
\begin{align*}
\left\|\begin{bmatrix}
\mathbf{M} & \mathbf{N}\\
{\mathbf{N}}^*& \mathbf{P}
\end{bmatrix}\right\|= \left\|{\big|X+(1-2\nu)X^*\big|}^2+{\big|X^*+(1-2\nu)X\big|}^2\right\|,
\end{align*}
\begin{align*}
w\left(\begin{bmatrix}
\mathbf{Q} & \mathbf{R}\\
\mathbf{T}& \mathbf{S}
\end{bmatrix}\right)= w\left({(X+(1-2\nu)X^*)}^2\right),
\end{align*}
and
\begin{align*}
c\left(\begin{bmatrix}
\mathbf{Q} & \mathbf{R}\\
\mathbf{T}& \mathbf{S}
\end{bmatrix}\right)= c\left({(X+(1-2\nu)X^*)}^2\right).
\end{align*}
Now, by \eqref{R.4.4.I.1} and Theorem \ref{T.00044} we deduce the desired result.
The inequalities \eqref{C.00044.I.2} also follow from \eqref{C.00044.I.1} by letting $\nu=\frac{1}{2}$.
\end{proof}
\begin{remark}\label{R.000211}
The inequalities \eqref{C.00044.I.2} in Corollary \ref{C.0002.30} are due to Abu-Omar and Kittaneh \cite[Theorem~2.4]{Abu-Omar.Kittaneh.RMJM.2015}.
\end{remark}
\textbf{Conflict of interest.}
The authors state that there is no conflict of interest.

\textbf{Data availability.}
Data sharing not applicable to the present paper as no data sets were generated or analyzed during the current study.

\textbf{Acknowledgement.}
The authors are grateful to Professor Pei Yuan Wu for his careful reading of the manuscript and for his comments and suggestions.
\bibliographystyle{amsplain}

\begin{thebibliography}{99}

\bibitem{Abu-Omar.Kittaneh.RMJM.2015} A. Abu-Omar and F. Kittaneh,
\textit{Upper and lower bounds for the numerical radius with an application to involution operators},
Rocky Mountain J.~Math. \textbf{45} (2015), no.~4, 1055--1064.

\bibitem{Abu-Omar.Kittaneh.LAA2019} A.~Abu-Omar and F.~Kittaneh,
\textit{A generalization of the numerical radius},
Linear Algebra Appl. \textbf{569} (2019), 323--334.

\bibitem{Aluthge.1966} A.~Aluthge,
\textit{Some generalized theorems on $p$-hyponormal operators},
Integr.~Equ.~Oper.~Theory \textbf{24} (1996), 497--501.

\bibitem{Ando.Li.2010} T.~Ando and C.~K.~Li,
\textit{Operator radii and unitary operators},
Oper.~Matrices \textbf{4}(2) (2010), 273--281.

\bibitem{Bhunia.Dragomir.Moslehian.Paul} P.~Bhunia, S.~S.~Dragomir, M.~S.~Moslehian and K.~Paul,
\textit{Lectures on numerical radius inequalities},
Infosys Science Foundation Series in Mathematical Sciences. Springer, 2022.

\bibitem{Conde.Sababheh.Moradi.2022} C.~Conde, M.~Sababheh and H.~R.~Moradi,
\textit{Some weighted numerical radius inequalities},
arXiv:2204.07620v1.

\bibitem{G.R} K.~E.~Gustafson and D.~K.~M.~Rao,
\textit{Numerical range. The field of values of linear operators and matrices},
Universitext. Springer-Verlag, New York, 1997.

\bibitem{Haagerup.Harpe.1992} U.~Haagerup and P.~de la Harpe,
\textit{The numerical radius of a nilpotent operator on a Hilbert space},
Proc.~Amer.~Math.~Soc. \textbf{115} (1992), no.~2, 371--379.

\bibitem{Halmos.1982} P.~R.~Halmos,
\textit{A Hilbert Space Problem Book},
2nd ed., Springer, New York, 1982.

\bibitem{Heinz.1951} E.~Heinz,
\textit{Beitr\"{a}ge zur St\"{o}rungstheorie der Spektralzerlegung},
Math.~Ann. \textbf{123} (1951), 415--438.

\bibitem{Holbrook.1968} J.~A.~R.~Holbrook,
\textit{On the power-bounded operators of Sz.~-Nagy and C.~Foia\c{s}},
Acta Sci.~Math.~(Szeged) \textbf{29} (1968), 299--310.

\bibitem{Holbrook.1971} J.~A.~R.~Holbrook,
\textit{Inequalities governing the operator radii associated with unitary $\rho$-dilations},
Michigan Math.~J. \textbf{18} (1971), 149--159.

\bibitem{Kittaneh.1988} F.~Kittaneh,
\textit{Notes on some inequalities for Hilbert space operators},
Publ.~Res.~Inst.~Math.~Sci. \textbf{24} (1988), 283–-293.

\bibitem{Kittaneh.2003} F.~Kittaneh,
\textit{A numerical radius inequality and an estimate for the numerical
radius of the Frobenius companion matrix},
Studia Math. \textbf{158} (2003), 11--17.

\bibitem{Kittaneh.2005} F.~Kittaneh,
\textit{Numerical radius inequalities for Hilbert space operators},
Studia Math. \textbf{168} (2005), 73-–80.

\bibitem{K.Z.JCAM.2023} F.~Kittaneh and A.~Zamani,
\textit{Bounds for $\mathbb{A}$-numerical radius based on an extension of $A$-Buzano inequality},
J.~Comput.~Appl.~Math. \textbf{426} (2023), 115070.

\bibitem{K.Z.LAA.2024} F.~Kittaneh and A.~Zamani,
\textit{On the $\rho$-operator radii},
Linear Algebra Appl. \textbf{687} (2024), 132–-156.

\bibitem{M.Z.BJMA.2023} M.~Mabrouk and A.~Zamani,
\textit{An extension of the $a$-numerical radius on $C^*$-algebras},
Banach J.~Math.~Anal. \textbf{17}, 42 (2023).

\bibitem{Okubo.Ando.1975} K.~Okubo and T.~Ando,
\textit{Constants related to operators of class ${\mathcal{C}}_{\rho}$},
Manuscripta Math. \textbf{16}(4) (1975), 385--394.

\bibitem{Okubo.Ando.1976} K.~Okubo and T.~Ando,
\textit{Operator radii of commuting products},
Proc.~Amer.~Math.~Soc. \textbf{56} (1976), 203--210.

\bibitem{S.K.S.AFA.2022} A.~Sheikhhosseini, M.~Khosravi and M.~Sababheh,
\textit{The weighted numerical radius},
Ann.~Funct.~Anal. \textbf{13}(3) (2022).

\bibitem{Nagy.Foias.1966} B.~Sz.~-Nagy and C.~Foia\c{s},
\textit{On certain classes of power-bounded operators in Hilbert space},
Acta Sci.~Math.~(Szeged) \textbf{27} (1966), 17--25.

\bibitem{Nagy.Foias.Bercovici.Kerchy.2010} B.~Sz.~-Nagy, C.~Foia\c{s}, H.~Bercovici and L.~K\'{e}rchy,
\textit{Harmonic Analysis of Operators on Hilbert Space},
Second edition, Universitext, Springer, New York, 2010.

\bibitem{Wang.Wu.Gau.LAA.2010} K-Z.~Wang, P.~Y.~Wu and H-L.~Gau,
\textit{Crawford numbers of powers of a matrix},
Linear Algebra Appl. \textbf{433} (2010), 2243--2254.

\bibitem{Williams.1968} J.~P.~Williams,
\textit{Schwarz norms for operators},
Pacific J.~Math. \textbf{24} (1968), 181--188.

\bibitem{Wu.JPT.1997} P.~Y.~Wu,
\textit{Unitary dilations and numerical ranges},
J.~Operator Theory, \textbf{38} (2022), 25-–42.

\bibitem{W.G} P.~Y.~Wu and H-L.~Gau,
\textit{Numerical ranges of Hilbert space operators},
Vol.~179.~Cambridge University Press, 2021.

\bibitem{Yamazaki.2007} T.~Yamazaki,
\textit{On upper and lower bounds of the numerical radius and an equality condition},
Studia Math. \textbf{178} (2007), 83--89.

\bibitem{Z.LAA.2023} A.~Zamani,
\textit{The weighted Hilbert--Schmidt numerical radius},
Linear Algebra Appl. \textbf{675} (2023), 225--243.

\bibitem{Z.W.LAA.2021} A.~Zamani and P.~W\'{o}jcik,
\textit{Another generalization of the numerical radius for Hilbert space operators},
Linear Algebra Appl. \textbf{609} (2021), 114--128.

\end{thebibliography}

\end{document}